\documentclass{amsart}

\usepackage{amsmath}
\usepackage{amsfonts}
\usepackage{amssymb}
\usepackage{amsthm}
\usepackage{diagrams}
\usepackage{parskip}
\usepackage{verbatim}
\usepackage{tikz}
\providecommand{\abs}[1]{\left\lvert#1\right\rvert}

\newtheorem{theorem}{Theorem}[section]
\newtheorem{proposition}[theorem]{Proposition}
\newtheorem{lemma}[theorem]{Lemma}

\theoremstyle{definition}
\newtheorem{definition}[theorem]{Definition}

\newtheorem{example}[theorem]{Example}

\theoremstyle{remark}

\providecommand{\Nn}{\mathbb{N}}

\providecommand{\Zz}{\mathbb{Z}}

\providecommand{\OP}{\operatorname}

\addtocontents{toc}{\setcounter{tocdepth}{1}}

\begin{document}
\title{Topological Mixing Properties of Rank-One Subshifts}
\author{Su Gao and Caleb Ziegler}
\address{Department of Mathematics, University of North Texas, 1155 Union Circle \#311430, Denton, TX 76203, USA}
\email{sgao@unt.edu}
\email{caleb.ziegler@unt.edu}

\date{\today}
\thanks{The first author acknowledges the US NSF grants DMS-1201290 and DMS-1800323 for the support of his research. Most results in this paper appeared as a part of the second author's PhD dissertation submitted to the University of North Texas in 2018.}

\begin{abstract}
We study topological mixing properties and the maximal equicontinuous factor of rank-one subshifts as topological dynamical systems. We show that the maximal equicontinuous factor of a rank-one subshift is finite. We also determine all the finite factors of a rank-one shift with a condition involving the cutting and spacer parameters. For rank-one subshifts with bounded spacer parameter we completely characterize weak mixing and mixing. For rank-one subshifts with unbounded spacer parameter we prove some sufficient conditions for weak mixing and mixing. We also construct some examples showing that the characterizations for the bounded spacer parameter case do not generalize to the unbounded spacer parameter case.
\end{abstract}
\maketitle \thispagestyle{empty}


\section{Introduction}
Rank-one transformations have been extensively studied since their introduction by Chacon \cite{Chacon} in 1965. As a result, Ferenczi wrote a survey \cite{Ferenczi} summarizing many results and systematically studying many of the different definitions of rank-one transformations that appeared in the literature. Many measure theoretic rank-one transformations could be shown to satisfy each of the different definitions. However, the constructive symbolic definition seemed to behave differently from other definitions, particularly with respect to odometers. This distinction led to further study of the constructive symbolic definition, such as by \cite{AdamsFerencziPetersen}, \cite{Danilenko1}, and \cite{Danilenko2}.

Because the constructive symbolic definition works with a shift space, it was natural to study systems coming from the constructive symbolic definition in the setting of topological dynamics. This led to the definition of rank-one subshifts, which was first studied by the first author and Hill in \cite{GaoHill}, where they gave a characterization for the topological isomorphism relation of rank-one subshifts based on the cutting and spacer parameters. Many rank-one subshifts carry a uniquely ergodic measure, and therefore can be also viewed as a rank-one transformation. To avoid confusion, whenever we refer to a rank-one subshift we will view it as a topological dynamical system unless we specify otherwise. This paper looks at many of the main areas of research for rank-one transformations in the measure-theoretic context and attempts to transfer them over to the case of rank-one subshifts in the topological context.

One major area of study for rank-one transformations concerns mixing properties. Indeed, the motivation for the original rank-one transformation constructed by Chacon \cite{Chacon} was to build a measure preserving transformation that is weakly mixing but not mixing. Other papers that looked at mixing properties of rank-one transformations include \cite{Adams}, \cite{AdamsFriedmanSilva}, \cite{BaylessYancey}, and \cite{Ornstein}. Of particular interest are the papers which attempted to classify mixing properties; the first author and Hill \cite{HillGao} classified when a rank-one transformation is weakly mixing in the canonically bounded case, and Creutz and Silva \cite{CrSil1} \cite{CrSil2} classified when a rank-one transformation is mixing based on the ergodicity of the sequence of spacer parameters.

In this paper we study topological mixing properties for rank-one subshifts. For topological weak mixing, we have the following complete classification for rank-one subshifts with bounded spacer parameter.

\begin{theorem}\label{mainweaklymixingtheorem}
Let $(X,T)$ be a rank-one subshift given by the cutting parameter $(q_n)$ and the spacer parameter $(a_{n,i})$. Let $(v_n)$ be the corresponding rank-one generating sequence. Suppose there is $B>0$ such that $a_{n,i}<B$ for all $n\in\Nn$ and $1\leq i<q_n$. Then $(X,T)$ is topologically weakly mixing iff for any interger $p>1$ and $n\in \Nn$, there are $m\geq n$ and $1\leq i< q_m$ such that $p\!\!\not|(\abs{v_n}+a_{m,i})$.
\end{theorem}

This mirrors the results of \cite{HillGao} in giving an explicit characterization of weak mixing in terms of the spacer parameters. In general, topological weak mixing neither implies, nor is implied by, weak mixing in the measure-theoretic sense; but in the case of canonically bounded rank-one subshifts, it turns out that they are topologically weakly mixing exactly when they are weakly mixing as a rank-one transformation. 

We also study topological weak mixing on rank-one subshifts with unbounded spacer parameter. We prove the following sufficient condition for this case.

\begin{theorem}\label{mainmeasureoneweakly}
Let $(X,T)$ be a rank-one subshift given by the cutting parameter $(q_n)$ and the spacer parameter $(a_{n,i})$. If the set $\{a_{n, i}:n\in\Nn, 1\leq i<q_n\}$ is a subset of $\Nn$ with density greater than $\frac{1}{2}$, then $(X,T)$ is topologically weakly mixing. 
\end{theorem}

We also study topological mixing for rank-one subshifts and obtain the following results. 

\begin{theorem}\label{mainmixingneverholds}
Let $(X,T)$ be a rank-one subshift given by the cutting parameter $(q_n)$ and the spacer parameter $(a_{n,i})$. Suppose there is $B>0$ such that $a_{n,i}<B$ for all $n\in\Nn$ and $1\leq i<q_n$. Then $(X,T)$ is not topologically mixing.
\end{theorem}

\begin{theorem}\label{mainmeasureoneweakly}
Let $(X,T)$ be a rank-one subshift given by the cutting parameter $(q_n)$ and the spacer parameter $(a_{n,i})$. If the set $\{a_{n, i}:n\in\Nn, 1\leq i<q_m\}$ is co-finite, then $(X,T)$ is topologically mixing. 
\end{theorem}

We produce several other sufficient conditions for topological mixing for rank-one subshifts with unbounded spacer parameter. We also give some examples to show that a complete characterization for weak mixing or mixing can be subtle for rank-one subshifts with unbounded spacer parameter.

Finally, it is a well-known fact in topological dynamical systems that for a large class of minimal dynamical systems, a dynamical system is weakly mixing iff it has trivial maximal equicontinuous factor (see a general reference such as \cite{Auslander}). A natural problem then is whether we can classify the maximal equicontinuous factor for rank-one subshifts. In previous classifications for Toeplitz systems by Williams \cite{Williams} and for generalizations of Toeplitz systems by Downarowicz \cite{Downarowicz}, the maximal equicontinuous factor was found to be odometers.

We classify exactly which odometers can be factors of rank-one subshifts and completely classify the maximal equicontinuous factor for rank-one subshifts via the following theorem.

\begin{theorem}\label{mainmainmef}
Let $(X,T)$ be a rank-one subshift given by the cutting parameter $(q_n)$ and the spacer parameter $(a_{n,i})$. Let $(v_n)$ be the corresponding rank-one generating sequence. Then the following hold.
\begin{enumerate}
\item[\rm (1)] If $(X,T)$ has unbounded spacer parameter, then the maximal equicontinuous factor is trivial.
\item[\rm (2)] If $(X,T)$ has bounded spacer parameter, then the maximal equicontinuous factor is the largest finite factor of the system. In particular, this finite factor is $\Zz/p_{\max}\Zz$, where $p_{\max}$ is the largest integer $p$ for which there is $n\in\Nn$ such that for all $m\geq n$ and $1\leq i< q_m$, $p|(\abs{v_n}+a_{m,i})$.
\end{enumerate}
\end{theorem}

In particular, infinite odometers cannot be factors of rank-one subshifts.

The rest of the paper is organized as follows. In \S\ref{preliminaries} we give the basic definitions and properties of rank-one subshifts. In \S\ref{expectedness}, we develop the main tools in the study of rank-one subshifts and prove the technical results on the expectedness structure that will be used for the rest of the paper.
In \S\ref{MEF} through \S\ref{mixing}, we proved the main results on the maximal equicontinuous factors, weakly mixing, and mixing.

\section{Definitions and Preliminaries\label{preliminaries}}

\subsection{Topological dynamics}
We will work with standard definitions in topological dynamics. For us, a {\em topological dynamical system} is a pair $(X,T)$, where $X$ is a compact metric space and $T:X\to X$ is a continuous map. 

Let $(X,T)$ and $(Y,S)$ be topological dynamical systems. We say $(Y,S)$ is a {\em factor} of $(X,T)$ if there is a continuous, onto map $\varphi:X\to Y$, so that $\varphi\circ T=S\circ\varphi$. We call $\varphi$ the {\em factor map}. 

Let $d_Y$ be a compatible metric on $(Y,S)$. We say $(Y,S)$ is {\em equicontinuous} if for any $\epsilon>0$, there is some $\delta>0$ so that for any $y_1,y_2\in Y$, whenever $d(y_1,y_2)<\delta$,  we have $d_Y(S^l(y_1),S^l(y_2))<\epsilon$ for any $l\in \Zz$. 

We say that $(Y,S)$ is the {\em maximal equicontinuous factor} of $(X,T)$ if $(Y,S)$ is a factor of $(X,T)$, $(Y,S)$ is equicontinuous, and for any $(Y',S')$ which is an equicontinuous factor of $(X,T)$, $(Y',S')$ is also a factor of $(Y,S)$. In particular, if $\varphi$ is the factor map from $(X,T)$ to $(Y,S)$ and $\psi$ is the factor map from $(X,T)$ to $(Y',S')$, then there is a factor map $\theta$ from $(Y,S)$ so $(Y',S')$ so that the following diagram commutes:
\begin{diagram}
(X,T) & & \\
\dTo^{\varphi} &\rdTo^{\psi} & \\
(Y,S) & \rDashto^{\theta} & (Y',S')
\end{diagram} 

It is well-known that any topogical dynamical system has a maximal equicontinuous factor (c.f. \cite{EllisGottschalk}). However, it is not necessarily easy to determine what the maximal equicontinuous factor is for a given topological dynamical system. We will do this in \S\ref{MEF} for rank-one subshifts.

We will also study the properties of topological weak mixing and topological mixing for rank-one subshifts. Since we only work in the topological setting, we will omit the modifier ``topological" throughout the paper. 

Recall that a topological dynamical system $(X,T)$ is {\em weakly mixing} if for any non-empty open sets $U,V,W,Z\subseteq X$, there is some $l\in\Zz$ so that $T^l(U)\cap V\neq \emptyset$ and $T^l(W)\cap Z\neq \emptyset$. $(X,T)$ is {\em mixing} if for any non-empty open sets $U,V\subseteq X$, there is an $L\in \Nn$, so that for any $l\geq L$, $T^l(U)\cap V\neq \emptyset$.

\subsection{Rank-one subshifts} The topological dynamical systems we study in this paper will be rank-one subshifts. 

In general, a {\em subshift} is a topological dynamical system $(X,T)$ where $X$ is a closed subspace of $A^\Zz$ for some discrete space $A$ and $T$ is the left shift map given by
$$T(x)(n)=x(n+1)$$
for $x\in X\subseteq A^\Zz$ and $n\in\Zz$.  In this paper we will be working with $A=2=\{0,1\}$ unless noted otherwise.

We fix some notation. We let $A^{<\omega}$ denote the set of all finite words over the alphabet $A$. If $x\in X$ and $n_1<n_2\in \Zz$, we let $x[n_1,n_2]$ denote the word $x(n_1)x(n_1+1)...x(n_2-1)x(n_2)$. For a finite word $w$, we denote the {\em length} of $w$ by $\abs{w}$. In particular, $\abs{x[n_1,n_2]}=n_2-n_1+1$.

The topology on $X$ is the subspace topology coming from the product topology on $A^\Zz$. It follows that the basic open subsets of $X$ are of the form 
$$U_{\alpha,k}=\{x\in X:x[k,k+\abs{\alpha}-1]=\alpha\}$$ 
for some word $\alpha\in A^{<\omega}$ and $k\in\Zz$.

Note that the shift map is well defined on the entire ambient space $A^\Zz$ as an autohomeomorphism. It follows that $X$ is a closed invariant subspace of $A^\Zz$, and $T$ is an autohomeomorphism of $X$.  

Next we recall the definition of our key concept studied in this paper, rank-one subshift.

Let $(q_n)$ be a sequence of natural numbers with each $q_n>1$. Let $(a_{n,i})$ be a doubly-indexed sequence of natural numbers where $n$ ranges over all natural numbers and $1\leq i<q_n$. 
The {\em rank-one generating sequence} $(v_n)$ given by {\em cutting parameter} $(q_n)$ and {\em spacer parameter} $(a_{n,i})$ is defined inductively by
$v_0=0$ and 
$$v_{n+1}=v_{n}1^{a_{n,1}}v_n1^{a_{n,2}}v_n...v_n1^{a_{n,q_{n}-1}}v_n $$
for all $n\in\Nn$. Note that in a rank-one generating sequence all words start and end with $0$ and each $v_n$ is an initial segment of $v_{n+1}$. This allows us to define
the {\em infinite rank-one word} $V$ to be the limit of the $v_n$, i.e. for each $k\in \Nn$, let $V(k)=v_n(k)$ for any $n$ such that $\abs{v_n}> k$. Finally, the {\em rank-one subshift} $(X_V, T)$ given by the infinite rank-one word $V$ is defined by
$$ X_V=\{ x\in X: \mbox{every finite subword of $x$ is a subword of $V$}\}. $$

The terminology of cutting and spacer parameters is inspired by the cutting-and-stacking construction that is used to define rank-one transformations in the measure-theoretic sense. We will not go into details of this construction, but will just note that 
the cutting parameter $q_n$ is the number of copies of $v_n$ that are used to construct $v_{n+1}$, and the $a_{n,i}$ specify the numbers of 1s inserted in between copies of $v_n$. We will thus refer to these 1s as {\em spacers}. It is useful to note that
$x\in X_V$ iff every subword of $x$ is a subword of $v_n$ for some $n$ iff every subword of $x$ is a subword of $v_n$ for sufficiently large $n$.

We say that the rank-one subshift has  {\em bounded spacer parameter} if there is a constant $B$ so that for all $n,i$, we have $a_{n,i}\leq B$. Otherwise, we say that it has {\em unbounded spacer parameter}. A rank-one subshift with bounded spacer parameter is a minimal dynamical system. A rank-one subshift with unbounded spacer parameter has exactly one fixed point $1^\Zz$.

\subsection{Some basic facts} We identify some specific elements in a rank-one subshift.

Let $(v_n)$ be a rank-one generating sequence and let $V=\lim_n v_n$ be the corresponding infinite rank-one word. Then $V$ is of the form
$$ V=v_n1^{k_0}v_n 1^{k_1} v_n\cdots\cdots $$
with $k_0, k_1, \dots\in\Nn$. If $V$ is periodic then the rank-one subshift generated is finite. We regard this the degenerate case. 

Observe that each $v_n$ is also an end segment of $v_{n+1}$. This allows us to define a {\em dual} infinite rank-one word $V^*$ as the {\em dual} limit of the $v_n$. Thus $V"$ is of the form
$$ V^*=\cdots\cdots v_n1^{l_1}v_n1^{l_0}v_n $$
with $l_0, l_1,\dots\in \Nn$. More formally, $V^*: -\Nn\to 2$ where for each $k\in\Nn$, $V(-k)=v_n(\abs{v_n}-k-1)$ for any $n$ such that $\abs{v_n}>k$.

It is easy to see that if a natural number $a\in\Nn$ occurs infinitely often in the spacer parameter sequence $(a_{n,i})$, then $V^*1^aV\in X$. This is because, every finite subword of $V^*1^aV$ is necessarily a subword of $v_n1^av_n$ for all sufficiently large $n$, which by our assumption is a finite subword of $v_{n+1}$ for infinitely many $n$. The following lemmas are immediate. 

\begin{lemma}\label{boundedVpoints}
Let $(X,T)$ be a rank-one subshift with bounded spacer parameter. Then there is $a\in\Nn$ such that $V^*1^aV$ in $X$. Moreover, if $(X, T)$ is infinite, then there are at least two values $a,a'\in \Nn$ such that $V^*1^aV,V^*1^{a'}V\in X$.
\end{lemma}

\begin{lemma}\label{infiniteones}
Let $(X,T)$ be a rank-one subshift with unbounded spacer parameter. Then the following infinite words are elements of $X$:
$$1^\Zz;\  1^{-\Nn}V; \ V^*1^{\Nn}.$$
Moreover, these are the only forms of infinite words with infinitely many consecutive $1$s.
\end{lemma}

Unless we specify otherwise, all rank-one subshifts we work with will be infinite.

\section{The Combinatorics of Expectedness \label{expectedness}}

\subsection{Expectedness}

One key concept in the study of rank-one subshifts is the notion of expectedness, which was defined in \cite{GaoHill}. We recall this notion and some facts.

If $(v_n)$ is a rank-one generating sequence, we have
$$ v_{n+1}=v_n1^{a_{n,1}}v_n1^{a_{n,2}}v_n\cdots 1^{a_{n,q_n-1}}v_n. $$
Each of the demonstrated occurrence of $v_n$ in this expression is called an {\em expected occurrence}. There might be unexpected occurrences of $v_n$ in this expression that occur as a subword of $v_n1^k v_n$ (where $1^k$ occurs in $v_n$), but it is important that we work with expected occurrences of $v_n$ when we consider the combinatorics of rank-one words.

More generally, for any $m>n$, $v_m$ can also be written as
$$ v_m=v_n1^{a_1}v_n1^{a_2}v_n\cdots v_n1^{a_t}v_n. $$
Note that for each $1\leq j\leq t$, there is some $n\leq n'<m$ and $1\leq i<q_{n'}$ such that $a_j=a_{n',i}$, i.e., the indices appeared in the expression all come from the spacer parameter in between level $n$ and level $m$. We refer to the demonstrated occurrences of $v_n$ in this expression also as {\em expected occurrences}.

When we write the infinite rank-one word $V$ in the form
$$ V=v_n1^{k_0}v_n1^{k_1}v_n\dots, $$
we again call each demonstrated occurrence of $v_n$ an {\em expected occurrence}. Here the indices demonstrated all come from the spacer parameter above level $n$, i.e., for each $j\in\Nn$ there is $n'>n$ and $1\leq i<q_{n'}$ such that $k_j=a_{n',i}$.

It was shown in \cite{GaoHill} that each non-$1^\Zz$ element of an infinite rank-one subshift can also be decomposed uniquely into expected occurrences of $v_n$ with spacers in between. To be precise, if $(X, T)$ is a rank-one subshift and $x\in X\!\setminus\!\{1^\Zz\}$, then there is a unique way to write $x$ in the form
$$ x=\cdots\cdots v_n1^{k_{-1}}v_n1^{k_0}v_n1^{k_1}v_n\cdots\cdots $$
for any $n\in\Nn$. The demonstrated occurrences of $v_n$ in $x$ are called {\em expected occurrences}. In this unique expression the indices again come from the spacer parameter above level $n$, i.e., for each $j\in\Zz$ there is $n'>n$ and $1\leq i<q_{n'}$ such that $k_j=a_{n',i}$.

It is easy to see that all these notions of expected occurrence cohere with each other. For instance, if there is an expected occurrence of $v_m$ in $V$ (or in any $x\in X$) and an expected occurrence of $v_n$ in $v_m$, then this occurrence of $v_n$ in $V$ (or in $x\in X$) is expected. 

It was also shown in \cite{GaoHill} that the sets of the form
$$ E_{n,k}=\{x\in X: \mbox{$x$ has an expected occurrence of $v_n$ beginning at position $k$}\} $$
where $n\in\Nn$ and $k\in\Zz$, generate the topology of $X\!\setminus\!\{1^\Zz\}$.

In the following, we collect some basic facts about the sets $E_{n,k}$ for our use in the rest of this paper.

\begin{proposition}[\cite{GaoHill}]\label{Evnkfacts}
Let $(X,T)$ be a rank-one subshift generated by $(v_n)$. Then the following hold.
\begin{enumerate}
\item Each $E_{n,k}$ is clopen.\label{Evnkclopen}
\item For any $n\in\Nn$ and $k,l\in \Zz$, $T^l(E_{n,k})=E_{n,k-l}$.\label{Evnkshift} 
\item For any $n\in \Nn$, there is a constant $C$ and finitely many words $\alpha_1, \dots,\alpha_r$ with $\abs{\alpha_j}\leq C$ for each $1\leq j\leq r$, so that for any $x\in X$ and $k\in \Zz$, $x\in E_{n,k}$ iff $x[k,k+\abs{\alpha_j}-1]=\alpha_j$ for some $1\leq j\leq r$.\label{Evnkdetermined}
\item For any open set $U\subseteq X$, there is some $n\in\Nn$ and $k\in\Zz$ so that $E_{n,k}\subseteq U$.\label{Evnkcontain}
\item If $(X, T)$ has bounded spacer parameter, then $\{E_{n,k}:n\in\Nn,k\in\Zz\}$ is a subbasis for the topology of $X$.
\item If $(X,T)$ has unbounded spacer parameter, then $\{E_{n,k}:n\in\Nn,k\in\Zz\}\cup \{U_{1^n,k}:n\in\Nn,k\in\Zz\}$ is a subbasis for the topology of $X$.
\end{enumerate}
\end{proposition}


\subsection{Blocks} We introduce a new concept to facilitate our study of the combinatorics of the expected occurrences of $v_n$.

\begin{definition}
Let $n\in\Nn$. A finite word $\alpha$ is called an {\em $n$-block} if $\abs{\alpha}\geq \abs{v_n}$, and there are $t<s\in\Nn$ such that $\alpha=V[t,s]$ and there are expected occurrences of $v_n$ in $V$ starting at positions $t$ and $s+1$.
\end{definition}

In general, any $n$-block is of the form
$$ \alpha=v_n1^{k_0}v_n 1^{k_1}\cdots v_n 1^{k_r}, $$
where each demonstrated occurrence of $v_n$ comes from an expected occurrence of $v_n$ in $V$. We will refer to these occurrences of $v_n$ also as {\em expected occurrences} of $v_n$ in $\alpha$. 

\begin{lemma}\label{blockequivalence} Let $\alpha$ be a finite word and $n\in\Nn$. Then the following are equivalent:
\begin{enumerate}
\item[(i)] $\alpha$ is an $n$-block.
\item[(ii)] There is some $x\in X$ and $t<s\in\Zz$ such that $\alpha=x[t,s]$ and $x\in E_{n,t}\cap E_{n,s+1}$.
\item[(iii)] For any $x\in X\!\setminus\!\{1^\Zz\}$ there are $t<s\in\Zz$ such that $\alpha=x[t,s]$ and $x\in E_{n,t}\cap E_{n,s+1}$.
\end{enumerate}
\end{lemma}

\begin{proof} We show (i)$\Rightarrow$(iii)$\Rightarrow$(ii)$\Rightarrow$(i). 

For (i)$\Rightarrow$(iii), suppose $\alpha$ is an $n$-block. Let $m>n$ be large enough so that the occurrence of $\alpha$ in $V$ is included in the first occurrence of $v_m$ in $V$. Then all the occurrences of $v_n$ in $\alpha$ are expected occurrences in $v_m$. Now let $x\in X\!\setminus\!\{1^\Zz\}$ be arbitrary. Then there is at least one expected occurrence of $v_m$ in $x$. As a subword of $v_m$, $\alpha$ thus occurs in $x$. Let $t$ be the starting position of $\alpha$ and $s$ be the ending position. Then the occurrences of $v_n$ at $t$ and $s+1$ are both expected. Thus $x\in E_{n,t}\cap E_{n,s+1}$.

It is obvious that (iii)$\Rightarrow$(ii).

To see (ii)$\Rightarrow$(i), let $t<s\in\Zz$, $x\in E_{n,t}\cap E_{s+1}$, and $\alpha=x[t,s]$. Let $m$ be large enough that $\abs{v_m}\geq\abs{\alpha}$. Each expected occurrence of $v_n$ in $x$ is contained in exactly one expected occurrence of $v_m$ in $x$, so we can find at most two consecutive expected occurrences of $v_m$ in $x$, with spacers in between, which contain the occurrence of $\alpha v_n$. In particular, $\alpha v_n$ is a subword of $v_m1^av_m$ for some $a=a_{m',i}$ where $m'\geq m$ and $1\leq i<q_{m'}$. Since each expected occurrence of $v_n$ in $\alpha$ comes from an expected occurrence of $v_n$ in $x$, it is still expected in $v_m1^av_m$. Now, we can find an expected occurrence of $v_{m'}$ in $V$, which gives rise to an occurrence of $v_m1^av_m$ as a subword of $V$ where each occurrence of $v_m$ is expected. It follows that all the expected occurrences of $v_n$ in $\alpha v_n$, while are expected within $v_m1^av_m$, are also expected in $V$. This shows that $\alpha$ is an $n$-block.
\end{proof}

The following lemma is an immediate corollary. 
\begin{lemma}\label{vnblockswitness}
Let $(X,T)$ be an infinite rank-one subshift generated by $(v_n)$. Then for any $n\in\Nn$ and $k<l\in \Zz$, we have $E_{n,k}\cap E_{n,l}\neq \emptyset$ iff there is some $n$-block $\alpha$ with $\abs{\alpha}=l-k$.
\end{lemma}
\begin{proof}
Let $k<l$ and suppose $E_{n,k}\cap E_{n,l}\neq \emptyset$. Let $x\in E_{n,k}\cap E_{n,l}$. By Lemma~\ref{blockequivalence} (ii)$\Rightarrow$(i), $\alpha=x[k,l-1]$ is an $n$-block with $\abs{\alpha}=l-k$.

Conversely, if $\alpha$ is an $n$-block with $\abs{\alpha}=l-k$, then by Lemma~\ref{blockequivalence} (i)$\Rightarrow$(ii) we get an $x\in E_{n,t}\cap E_{n,s+1}$ with $\alpha=x[t,s]$. It follows that $l-k=\abs{\alpha}=s-t+1$, and $l=(s+1)-(t-k)$. By Proposition~\ref{Evnkfacts}(\ref{Evnkshift}), $T^{t-k}(x)\in E_{n,k}\cap E_{n,l}$ and so $E_{n,k}\cap E_{n,l}\neq\emptyset$.
\end{proof}

Therefore, by understanding the subwords of $V$ that are $n$-blocks, we can understand which intersections of the $E_{n,k}$ will be empty and thus understand the topological structure of the rank-one subshift.

\subsection{Blocks with bounded spacer parameter}

In this subsection we study the topological structure of a rank-one subshift with bounded spacer parameter. 

We first introduce a piece of notation. For $m\geq n$, let $q_{n}^{(m)}$ denote the number of expected occurrences of $v_n$ within $v_m$. In fact, $q_{n}^{(n)}=1$ and for $m>n$,
$$ q_{n}^{(m)}=\prod_{n\leq n'<m} q_{n'}. $$

The following squence of lemmas study $n$-blocks with $q_{n}^{(m)}$ many expected occurrences of $v_n$.

\begin{lemma}\label{vmsamelength}
Let $m\geq n$. Let $\alpha$ be an $n$-block of the form
$$ \alpha=v_n1^{a_1}v_n1^{a_2}\cdots v_n1^{a_r}, $$
where $r=q_n^{(m)}$. Then there is $1\leq j\leq r$ such that $a_j=a_{m', i}$ for some $m'\geq m$ and $1\leq i<q_{m'}$ and 
$$ v_m=v_n1^{a_{j+1}}\cdots v_n1^{a_r}v_n1^{a_1}\cdots v_n1^{a_{j-1}}v_n. $$
\end{lemma}

\begin{proof} The lemma is trivial when $m=n$. We assume $m>n$. By Lemma~\ref{blockequivalence} let $x\in X$ contain an occurrence of $\alpha$ in which all expected occurrences of $v_n$ in $\alpha$ are expected in $x$. Since $\alpha$ contains $q_n^{(m)}$ many expected occurrences of $v_n$, the expected occurrences of $v_n$ in $\alpha$ are contained in at most two consecutive expected occurrences of $v_m$ in $x$.

If the expected $v_n$ in $\alpha$ are all contained within one expected occurrence of $v_m$ in $x$, then we necessarily have that $\alpha=v_m1^{a_r}$, where the next expected occurrence of $v_m$ in $x$ is to the immediate right of this occurrence of $\alpha$. It follows that $a_r=a_{m',i}$ for some $m'\geq m$ and $1\leq i<q_{m'}$ and the lemma holds.

If the occurrence of $\alpha$ is contained in two consecutive expected occurrences of $v_m$ with spacers in between, assume $\alpha$ contains $j<r$ expected $v_n$ which are contained in the first expected $v_m$. The second expected occurrence of $v_m$ starts with the $(j+1)$st expected occurrence of $v_n$ in $\alpha$. It follows that $a_j=a_{m',i}$ for some $m'\geq m$ and $1\leq i<q_{m'}$, and we obtain
$$ v_m=v_n1^{a_{j+1}}\cdots v_n1^{a_r}v_n1^{a_1}\cdots v_n1^{a_{j-1}}v_n $$
by a comparison of the two expected occurrences of $v_m$ with the occurrence of $\alpha$. 
\end{proof}

\begin{lemma}\label{singlespacerdifference}
Let $m\leq n$ and let $\alpha$ be an $n$-block with exactly $q_n^{(m)}$ many expected occurrences of $v_n$. Then $\abs{\alpha}-\abs{v_m}=a_{m',i}$ for some $m'\geq m$ and $1\leq i<q_{m'}$. 
\end{lemma}

\begin{proof} This follows immediately from Lemma~\ref{vmsamelength}.\end{proof}

\begin{lemma}\label{vklengthinblocks} 
Suppose the spacer parameter is bounded by $B>0$. Let $m\geq n$ such that $\abs{v_n}> B$. Let $\alpha$ be an $n$-block. Suppose $0\leq \abs{\alpha}-\abs{v_m}<\abs{v_n}$. Then $\alpha$ contains exactly $q_n^{(m)}$ many expected occurrences of $v_n$. 
\end{lemma}
\begin{proof} If $m=n$, then it is easy to see that $\alpha$ must contain exactly one expected occurrence of $v_n$. Now let $m>n$ and suppose by contradiction that $\alpha$ does not contain exactly $q_n^{(m)}$ many expected occurrence of $v_n$.

First, assume that $\alpha$ contains fewer than $q_n^{(m)}$ many expected occurrences of $v_n$. By Lemma~\ref{blockequivalence} let $x$ contain an occurrence of $\alpha$ where all expected occurrences of $v_n$ in $\alpha$ are expected in $x$. Then, we can extend $\alpha$ to a subword $\beta$ of $x$ so that $\beta$ is an $n$-block and $\beta$ contains exactly $q_n^{(m)}$ many expected occurrences of $v_n$. Note that $\abs{\beta}\geq \abs{\alpha}+\abs{v_n}\geq \abs{v_m}+\abs{v_n}$.
By Lemma~\ref{singlespacerdifference} and the boundedness of the spacer parameter, we have $\abs{\beta}-\abs{v_m}\leq B$. But then we have $\abs{v_m}+\abs{v_n}\leq \abs{\beta}\leq\abs{v_m}+B$, which would imply $\abs{v_n}\leq B$, a contradiction.

Now suppose $\alpha$ contains more than $q_n^{(m)}$ many expected occurrences of $v_n$. Then we can shrink $\alpha$ to an $n$-block $\gamma$ which contains exactly $q_n^{(m)}$ many expected occurrences of $v_n$. We have $\abs{\gamma}\leq \abs{\alpha}-\abs{v_n}<\abs{v_m}$. However, by Lemma~\ref{singlespacerdifference} $\abs{\gamma}\geq \abs{v_m}$, again a contradiction.
\end{proof}

\begin{lemma}\label{vkmisses}
Suppose the spacer parameter is bounded by $B>0$. Let $m\geq n$ be such that $\abs{v_n}> B$.  Let $0\leq d<\abs{v_n}$. Then we can have $E_{n,0}\cap E_{n,\abs{v_m}+d}\neq \emptyset$ only when $d\leq B$.
\end{lemma}

\begin{proof} Suppose $E_{n,0}\cap E_{n,\abs{v_m}+d}\neq\emptyset$. By Lemma~\ref{vnblockswitness} there is some $n$-block $\alpha$ with $\abs{\alpha}=\abs{v_m}+d$. By Lemma~\ref{vklengthinblocks}, $\alpha$ contains exactly $q_n^{(m)}$ many expected occurrences of $v_n$. Finally by Lemma~\ref{singlespacerdifference}, $d\leq B$.
\end{proof}

\subsection{Blocks of different lengths} In this subsection we study the possible differences between lengths of $n$-blocks. 

We will use the following technical lemma iteratively in our constructions. 

\begin{lemma}\label{mainconstruction}
Let $m\geq n$ and let $\alpha$ be an $n$-block of the form
$$ \alpha=v_n1^{a_1}\cdots v_n1^{a_r}, $$
where $r\leq q_n^{(m)}$. Suppose $1\leq j\leq r$ is such that $a_j=a_{m',i'}$ for some $m'\geq m$ and $1\leq i'<q_{m'}$. Letting
$$\alpha_0=v_n1^{a_1}\cdots v_n1^{a_{j-1}}v_n \mbox{ and }$$
$$\alpha_1=v_n1^{a_{j+1}}\cdots v_n1^{a_r}, $$
suppose $\alpha_0$ is an end segment of $v_m$ and $\alpha_1$ is an initial segment of $v_m$.
Then for any $1\leq i_0\leq i_1<q_m$, $\bar{m}\geq m+1$, and $1\leq \bar{\imath}<q_{\bar{m}}$, the following word $\beta$ is an $n$-block:
$$ \beta=\alpha_0\gamma_1 1^{a_{\bar{m},\bar{\imath}}} \gamma_0\alpha_1, $$
where
$$
\gamma_0= v_m1^{a_{m,1}}v_m\cdots v_m1^{a_{m,i_0-1}}, \mbox{ and} $$
$$ \gamma_1= 1^{a_{m,i_1+1}}v_m\cdots v_m1^{a_{m,q_m-1}}v_m.
$$
Moreover, $\beta$ contains at most $q_n^{(m+1)}$ many expected occurrences of $v_n$.
\end{lemma}

\begin{proof}
By our assumption, $\alpha_0\gamma_1$ is in fact a subword of 
$$ v_m1^{a_{m,i_1+1}}v_m\cdots v_m1^{a_{m,q_m-1}}v_m, $$
which is an end segment of $v_{m+1}$, and $\gamma_0\alpha_1$ is a subword of
$$ v_m1^{a_{m,1}}v_m\cdots v_m1^{a_{m,i_0-1}}v_m, $$
which is an initial segment of $v_{m+1}$. Thus $\beta$ is a subword of $v_{m+1}1^{a_{\bar{m},\bar{\imath}}}v_{m+1}$, which is in turn a subword of $v_{\bar{m}+1}$. This implies that $\beta$ is an $n$-block. 

Now the sum of the numbers of expected occurrences of $v_m$ in $\beta_0$ and $\beta_1$ is at most $q_m-1$, and thus the total number of expected occurrences of $v_n$ in $\beta$ is at most
$$ (q_m-1)q_n^{(m)}+r\ \leq \ q_mq_n^{(m)}\ =\ q_n^{(m+1)}. $$
\end{proof}

This lemma can be viewed as an inductive step in an iterative construction. In fact, we start with the $n$-block $\alpha=\alpha_0 1^{a_j}\alpha_1$ and construct the $n$-block $\beta=\beta_0 1^{a_{\bar{m},\bar{\imath}} }\beta_1$ if we let $\beta_0=\alpha_0\gamma_1$ and $\beta_1=\gamma_0\alpha_1$. Each time we replace the distinguished spacer parameter by spacer parameter of a higher level.

In the following proposition we state a sufficient condition for all numbers to be possible differences between lengths of $n$-blocks.

\begin{proposition}\label{onegap}
Let $n\in\Nn$. Suppose for infinitely many $m$, there are $1\leq i, j<q_m$ with $a_{m,i}-a_{m,j}=1$. Then for any $h\in \Nn$ there are $n$-blocks $\alpha$, $\beta$ such that $\abs{\alpha}-\abs{\beta}=h$. Moreover, for any $n$-block $\gamma$ and $h\in\Nn$ there are $n$-blocks $\alpha$, $\beta$ such that $\gamma$ is an initial segment of both $\alpha$ and $\beta$, and $\abs{\alpha}-\abs{\beta}=h$.
\end{proposition}
\begin{proof} We define $\alpha_k$ and $\beta_k$ by induction on $0\leq k\leq h$ so that $\abs{\alpha_k}-\abs{\beta_k}=k$. For $k=0$, let $m_0>n$ so that there are $1\leq i_0,j_0<q_{m_0}$ with $a_{m_0,j_0}-a_{m_0, i_0}=1$. Let $\alpha_0=\beta_0$ be an $n$-block with $q_n^{(m_0)}$ many expected occurrences of $v_n$. By Lemma~\ref{vmsamelength} $\alpha_0$ satisfies the assumption of Lemma~\ref{mainconstruction}, i.e., $\alpha_0=\alpha_{0,0}1^{a_j}\alpha_{0,1}$, where $\alpha_{0,0}$ is an end segment of $v_{m_0}$ and $\alpha_{0,1}$ is an initial segmant of $v_{m_0}$. 

In general, assume $\alpha_k$ and $\beta_k$ have been defined and they satisfy the assumption of Lemma~\ref{mainconstruction} for some $m_k>m_{k-1}$ such that there are
$1\leq i_k,j_k<q_{m_k}$ with $a_{m_k,j_k}-a_{m_k,i_k}=1$. Arbitrarily pick $\bar{m}_k\geq m_k+1$ and $1\leq \bar{\imath}_k<q_{\bar{m}_k}$. Construct $\alpha_{k+1}$ from $\alpha_k$ by applying Lemma~\ref{mainconstruction} using $a_{\bar{m}_k, \bar{\imath}_k}$ and with the $i_0,i_1$ of Lemma~\ref{mainconstruction} both set as $i_k$. Similarly, construct $\beta_{k+1}$ from $\beta_k$ by applying Lemma~\ref{mainconstruction} also using $a_{\bar{m}_k,\bar{\imath}_k}$ but with the $i_0, i_1$ of Lemma~\ref{mainconstruction} both set as $j_k$ instead. Then $\alpha_{k+1}$ and $\beta_{k+1}$ are $n$-blocks, and a comparison of their lenths shows that $\abs{\alpha_{k+1}}-\abs{\beta_{k+1}}=\abs{\alpha_k}-\abs{\beta_k}+a_{m_k,j_k}-a_{m_k,i_k}=\abs{\alpha_k}-\abs{\beta_k}+1$. To finish the inductive step, pick $m_{k+1}>m_k$ so that there are $1\leq i_{k+1}, j_{k+1}<q_{m_{k+1}}$ with $a_{m_{k+1}, j_{k+1}}-a_{m_{k+1},i_{k+1}}=1$. 

This finishes the indutive definition of $\alpha_k$ and $\beta_k$. The first part of the proposition is proved with $\alpha=\alpha_h$ and $\beta=\beta_h$. 

For the second part of the proposition, let $\gamma$ be any $n$-block. In the definition of $\alpha_0=\beta_0$ above we let $m_0$ be large enough so that $\gamma$ occurs in $v_{m_0}$ with all of its $r$ many expected occurrences of $v_n$ occur also expected in $v_{m_0}$. Let $\alpha_{0,0}$ be the end segment of $v_{m_0}$ starting with this occurrence of $\gamma$. Let $\alpha_{0,1}$ be the intitial segment of $v_{m_0}$ with $\alpha_{0,1}\alpha_{0,0}=v_{m_0}$. Let $a_j=a_{m_0,i_0}$ for any $1\leq i_0<q_{m_0}$. By Lemma~\ref{mainconstruction}, each $\alpha_k$ or $\beta_k$ has $\alpha_{0,0}$ as an initial segment, and therefore also has $\gamma$ as an initial segment.
\end{proof}

We can generalize this result by telescoping as follows.

\begin{proposition}\label{onegapdifferentlevels}
Let $n\in\Nn$. Suppose that for any $N$ there are $m,k\geq N$ and $1\leq i<q_{m}, 1\leq j<q_{k}$ with $a_{m, i}-a_{k,j}=1$. Then for any $h\in \Nn$ there are $n$-blocks $\alpha$, $\beta$ such that $\abs{\alpha}-\abs{\beta}=h$. Moreover, for any $n$-block $\gamma$ and $h\in\Nn$ there are $n$-blocks $\alpha$, $\beta$ such that $\gamma$ is an initial segment of both $\alpha$ and $\beta$, and $\abs{\alpha}-\abs{\beta}=h$.
\end{proposition}

\begin{proof} Let $(v_t)$ be the rank-one generating sequence. For all $l\in\Nn$, let $n<N_l<m_l,k_l<N_{l+1}$ be such that there are $1\leq i<q_{m_l}$ and $1\leq j<q_{k_l}$ with $a_{m_l,i}-a_{k_l,j}=1$. Then we consider a new rank-one generating sequence $(w_t)$ where $w_t=v_t$ for all $t\leq N_0$ and $w_{N_0+s}=v_{N_s}$ for all $s\in\Nn$. Then each of the spacer parameter for $(v_t)$ on the $m_l, k_l$ levels appears as a spacer parameter for $(w_t)$ on the level $N_0+l$. Thus the assumption of Proposition~\ref{onegap} is satisfied for $(w_t)$. Now if $\alpha$ and $\beta$ are $n$-blocks for $(w_t)$ then they are still $n$-blocks for $(v_t)$. Thus the proposition is proved.
\end{proof}

Next we turn to the more general case for the spacer parameter. For this we introduce a new number-theoretic concept.

\begin{definition}
Let $\{a_1,\dots,a_l\}$ be a finite set of natural numbers with at least two distinct elements. We define the {\em up-down gcd} to be the minimum value $d\geq 1$ achievable by a sum of the form $\sum_{i=1}^I (a_{i,+}- a_{i,-})$ for some $I\in \Nn$, where each $a_{i,+},a_{i,-}\in \{a_1,\dots,a_l\}$.
\end{definition}


It is obvious that the gcd of $\{a_1,\dots, a_l\}$ is a factor of the up-down gcd of $\{a_1,\dots, a_l\}$. The following lemma gives a characterization of the up-down gcd.

\begin{lemma}\label{updowngcd}
Let $\{a_1,\dots,a_l\}$ be a finite set of natural numbers with at least two distinct elements. Let $\{b_1,\dots,b_h\}$ be the set of differences from $\{a_1,\dots, a_l\}$, i.e. the $b_i$ are all values of the form $\abs{a_j-a_j'}$ for $1\leq j\neq j'\leq l$. Then the up-down gcd of $\{a_1,\dots,a_l\}$ is the gcd of $\{b_1,\dots,b_h\}$. 
\end{lemma}
\begin{proof}
Let $d$ be the gcd of $\{b_1,\dots,b_h\}$ and $d'$ be the up-down gcd of $\{a_1,\dots, a_l\}$. Write $d$ as a linear combination of the $b_i$. Note that replacing each $b_i$ by an appropriate $a_j-a_{j'}$ gives that $d$ can be written as a sum in the desired form. This shows that $d$ is greater than $d'$. On the other hand, by the definition of the up-down gcd, $d'$ is a linear combination of the $b_i$. Thus $d$ is a factor of $d'$. Hence $d=d'$.
\end{proof}


\begin{proposition}\label{updownachievable}
Let $n\in\Nn$. Assume $\{a_1,\dots,a_l\}$ is a finite subset of the spacer parameter with at least two distinct values and such that each $a_1,\dots, a_l$ occurs infinitely often.
Let $d$ be the up-down gcd of $\{a_1\dots, a_l\}$. Then for any $h\in \Nn$ there are $n$-blocks $\alpha$ and $\beta$ such that $\abs{\alpha}-\abs{\beta}=hd$. Moreover, for any $n$-block $\gamma$ and $h\in\Nn$ there are $n$-blocks $\alpha$ and $\beta$ such that $\gamma$ is an initial segment of both $\alpha$ and $\beta$, and $\abs{\alpha}-\abs{\beta}=hd$.
\end{proposition}
\begin{proof}
Without loss of generality assume that all $a_1,\dots, a_l$ are distinct and $l\geq 2$. Let $h\in\Nn$. Since $d$ is the up-down gcd, it follows that we can write $hd=\sum_{i=1}^{I}(a_{i,+}-a_{i,-})$ where $a_{i,+},a_{i,-}\in \{a_1,...,a_l\}$. We produce $n$-blocks $\alpha$ and $\beta$ so that $\abs{\alpha}-\abs{\beta}=hd$. Similar to the proofs of Propositions~\ref{onegap} and \ref{onegapdifferentlevels}, we construct $\alpha$ and $\beta$ by an induction in $I$ many steps, starting with $\alpha_0=\beta_0$, and making sure $\abs{\alpha_{i+1}}-\abs{\beta_{i+1}}=\abs{\alpha_i}-\abs{\beta_i}+a_{i,+}-a_{i,-}$. This is achieved by applying Lemma~\ref{mainconstruction} in a way so that a spacer parameter $a_{i,-}$ is omitted in the construction of $\alpha_{i+1}$ and a spacer parameter $a_{i,+}$ is omitted in the construction of $\beta_{i+1}$. Letting $\alpha=\alpha_I$ and $\beta=\beta_I$, we have $\abs{\alpha}-\abs{\beta}=hd$.
\end{proof}

\section{Maximal Equicontinuous Factors \label{MEF}}

\subsection{Finite factors of rank-one subshifts} In a rank-one subshift any non-$1^\Zz$ element generates a dense orbit. This implies that the only finite factors of a rank-one subshift are cyclic transformations $(\Zz/p\Zz, x\mapsto x+1\!\mod\! p)$ for $p\geq 1$. For brevity we will write this finite cyclic transformation as $\Zz/p\Zz$. Also, if the rank-one subshift has unbounded spacer parameter, then $1^\Zz$ is a fixed point, which implies that the only finite factor is trivial. 

Next we describe all finite factors of a rank-one subshift with bounded spacer parameter. We again assume that the rank-one subshift $(X,T)$ is given by the cutting parameter $(q_n)$ and spacer parameter $(a_{n,i})$ and is generated by $(v_n)$.

\begin{proposition}\label{Zzncondition}
Let $p\geq 2$ be an integer. A rank-one subshift $(X, T)$ with bounded spacer parameter has $\Zz/p\Zz$ as a factor iff there is some $n\in\Nn$, such that for all $m\geq n$, and all $1\leq i< q_m$, $p|(\abs{v_n}+a_{m,i})$.
\end{proposition}
\begin{proof} First assume $n\in\Nn$ is such that for all $m\geq n$ and $1\leq i<q_m$, $p|(\abs{v_n}+a_{m,i})$. Note that such $\abs{v_n}+a_{m_i}$ are precisely the lengths of an $n$-block with one expected occurrence of $v_n$. We define a factor map $\varphi: X\to \Zz/p\Zz$ based on the starting positions of the expected occurrences of $v_n$. Given any $x\in X$, arbitrarily pick a starting position $k$ of an expected occurrence of $v_n$ in $x$, and let $\varphi(x)=-k \!\mod\! p$. Our assumption guarantees that $\varphi(x)$ does not depend on the particular $k$ selected, as all $n$-blocks have lengths divisible by $p$. Since $\varphi(T(x))=\varphi(x)+1\!\mod\! p$, $\varphi$ is a factor map.

Conversely, let $\varphi:X\to \Zz/p\Zz$ be the factor map. By Proposition~\ref{Evnkfacts}(\ref{Evnkcontain}), there is some $E_{n,k}\subseteq \varphi^{-1}(\{0\})$. Now let $m\geq n$ and $1\leq i< q_m$. Consider an arbitrary $x\in X$. $x$ will necessarily contain an expected occurrence of $v_{m+1}$, and it follows that $x$ contains an occurrence of $v_n1^{a_{m,i}}v_n$ where the occurrences of $v_n$ are expected. Let $l$ be the starting position of this occurrence of $v_n1^{a_{m,i}}v_n$ in $x$. Then $x\in E_{n,l}\cap E_{n,l+\abs{v_n}+a_{m,i}}$, and so $T^{l-k}(x), T^{l-k+\abs{v_n}+a_{m,i}}(x)\in E_{n,k}$. Thus $\varphi(T^{l-k}(x))=\varphi(T^{l-k+\abs{v_n}+a_{m,i}}(x))=0$. But $\varphi(T^{l-k+\abs{v_n}+a_{m,i}}(x))=\varphi(T^{l-k}(x))+\abs{v_n}+a_{m,i}\!\mod\!p$, so $\abs{v_n}+a_{m,i}\equiv 0\!\mod\!p$.  
\end{proof}

There is a limitation on what finite factors are possible for a rank-one subshift with bounded spacer parameter.

\begin{proposition}\label{maxfinite}
Let $(X,T)$ be an infinite rank-one subshift with bounded spacer parameter. Then there is a largest $p_{\max}$ so that $\Zz/p_{\max}\Zz$ is a factor of $(X,T)$. Moreover, if $\Zz/p \Zz$ is any finite factor of $(X,T)$, then $p|p_{\max}$.
\end{proposition}
\begin{proof}
Let $C$ be a bound for the spacer parameter and suppose $(X,T)$ has a factor $\Zz/p\Zz$ where $p>C$. By Proposition \ref{Zzncondition}, there is some $n$ so that $p|(\abs{v_n}+a_{m,i})$ for all $m\geq n$ and $1\leq i< q_m$. Since $p>C\geq a_{m,i}$, there is only one possible value for $a_{m,i}$. Therefore, the $a_{m,i}$ are constant for $m\geq n$ and $1\leq i<q_m$. But then the infinite rank-one word is periodic, and $(X,T)$ cannot be infinite, a contradiction.

Now, if $\Zz/p\Zz$ and $\Zz/q\Zz$ are both finite factors of $(X, T)$, then Proposition~\ref{Zzncondition} gives that $\Zz/\mbox{lcm}(p,q)\Zz$ is a finite factor. Thus $p_{\max}$ is a multiple of all $p$ where $\Zz/p\Zz$ is a factor of $(X, T)$.
\end{proof}

Thus any rank-one subshift can have only finitely many finite factors. This implies that a rank-one subshift cannot have an infinite odometer as a factor, since an infinite odometer has infinitely many finite factors. This is in contrast to measure-theoretic rank-one transformations. It is well known that rank-one transformations can have infinite odometer factors.

\subsection{Partition proximality} We introduce a general concept that allows us to identify equicontinuous factors of rank-one subshifts.

\begin{definition}
Let $(X,T)$ be a topological dynamical system. We say that $X$ has {\em partition proximality} if there are finitely many points $x_1,...,x_p$ so that for any $x\in X$, there is some $x_i$, $1\leq i\leq p$, so that for any $\delta>0$, we can find $z_1,z_2\in X$ and some $l_1,l_2,l_3\in \Zz$, with $d(T^{l_1}(x),T^{l_1}(z_1))<\delta$, $d(T^{l_2}(x_i),T^{l_2}(z_2))<\delta$, and $d(T^{l_3}(z_1),T^{l_3}(z_2))<\delta$. We call the finitely many $x_1,...,x_p$ {\em reference points}.
\end{definition}

This property is of interest because of the following proposition.
\begin{proposition}\label{notafactorfinite}
Let $(X,T)$ be a topological dynamical system having partition proximality with $p$ many reference points. Let $(Y,S)$ be an equicontinuous factor of $(X,T)$. Then $Y$ is finite. In fact, $\abs{Y}\leq p$. 
\end{proposition}
\begin{proof}
Let $(X,T)$ have partition proximality with reference points $x_1,...,x_p$. Let $\varphi:X\to Y$ be the factor map. Let $x\in X$ and $x_i$ be the reference point for $x$ in the definition of partition proximality. We show $\varphi(x)=\varphi(x_i)$. 

Suppose $\varphi(x)\neq \varphi(x_i)$. Let $\epsilon=d_Y(\varphi(x),\varphi(x_i))>0$. Since $Y$ is equicontinuous, we can find some $\delta_Y>0$ so that for any $y_1,y_2\in Y$, whenever $d_Y(y_1,y_2)<\delta_Y$, we have $d_Y(S^l(y_1),S^l(y_2))<\frac{\epsilon}{3}$ for all $l\in\Zz$. Since $X$ and $Y$ are compact metric spaces, and $\varphi$ is continuous, we can find some $\delta>0$, so that for any $w_1,w_2\in X$, if $d_X(w_1,w_2)<\delta$, then $d_Y(\varphi(w_1),\varphi(w_2))<\delta_Y$.

Since $X$ has partition proximality, there are $z_1,z_2\in X$ and $l_1,l_2,l_3\in \Zz$ such that 
$$d_X(T^{l_1}(x),T^{l_1}(z_1)),\ d_X(T^{l_2}(x_i),T^{l_2}(z_2)),\ d_X(T^{l_3}(z_1),T^{l_3}(z_2))<\delta. $$
Thus
$$ d_Y(\varphi(T^{l_1}(x)),\varphi(T^{l_1}(z_1))), \ d_Y(\varphi(T^{l_2}(x_i)),\varphi(T^{l_2}(z_2))),\ d_Y(\varphi(T^{l_3}(z_1)),\varphi(T^{l_3}(z_2)))<\delta_Y.$$ 
By equicontinuity, we get 
$$d_Y(\varphi(x),\varphi(z_1)),\ d_Y(\varphi(x_i),\varphi(z_2)),\ d_Y(\varphi(z_1),\varphi(z_2))<\displaystyle\frac{\epsilon}{3}.$$ 
From triangle inequality, we get $d_Y(\varphi(x),\varphi(x_i))<\epsilon$, a contradiction.
\end{proof}



The next lemma characterizes the partition proximality property in rank-one subshifts in terms of $E_{n,k}$.

\begin{lemma}\label{Evnkproperty}
Let $(X,T)$ be a rank-one subshift with bounded spacer parameter. Let $p\geq 1$ be an integer and $x_1,\dots, x_p\in X$. Then $(X,T)$ has partition proximality with reference points $x_1,\dots, x_p$ iff for any $x\in X$ and sufficiently large $n\in\Nn$, there are $k_1,k_2,k_3\in \Zz$ and $1\leq i\leq p$ such that $x\in E_{n,k_1}$, $x_i\in E_{n,k_2}$, and both $E_{n,k_1}\cap E_{n,k_3}\neq \emptyset$ and $E_{n,k_2}\cap E_{n,k_3}\neq \emptyset$.
\end{lemma}
\begin{proof}
Suppose $(X,T)$ has partition proximality with reference points $x_1,...,x_p$. Let $B>0$ be an upper bound for the spacer parameter of $(X, T)$. Fix $x\in X$ and $n\in\Nn$. Let $x_i$ be the reference point witnessing partition proximality for $x$. By Proposition \ref{Evnkfacts}(\ref{Evnkdetermined}), there is $C>0$ and finitely many words $\alpha_1, \dots, \alpha_r$ with $\abs{\alpha_j}\leq C$ for all $1\leq j\leq r$, such that for any point $y\in X$ and $k\in \Zz$ with an occurrence of $v_n$ starting at position $k$ in $y$, $y\in E_{n,k}$ iff $y$ contains an occurrence of some $\alpha_j$ starting at position $k$. Let $\delta>0$ be small enough so that for any $y_1,y_2\in X$, $d(y_1,y_2)<\delta$ implies that $y_1$ and $y_2$ agree on a string of length at least $\abs{v_n}+B+C$ starting at position $0$, i.e. $y_1[0,\abs{v_n}+B+C]=y_2[0,\abs{v_n}+B+C]$. 

We claim that if $y_1[0,\abs{v_n}+B+C]=y_2[0,\abs{v_n}+B+C]$ then there is $0\leq k<\abs{v_n}+B$ such that $y_1, y_2\in E_{n,k}$. To see this, let $0\leq k<\abs{v_n}+B$ be such that $y_1$ has an expected occurrence of $v_n$ starting at position $k$. Such $k$ must exist since $y_1$ is covered by expected occurrences of $v_n$ with spacers in between, and the numbers of consecutive spacers are bounded by $B$. Since this occurrence of $v_n$ in $y_1$ is expected, there is $1\leq j\leq r$ such that $y_1[k, k+\abs{\alpha_j}-1]=\alpha_j$. By our assumption, $y_2[k,k+\abs{\alpha_j}-1]=\alpha_j$. Thus $y_2\in E_{n,k}$. This proves the claim.

Now, from partition proximality we get $z_1,z_2\in X$ and $l_1,l_2,l_3\in \Zz$ such that 
$$d(T^{l_1}(x), T^{l_1}(z_1)), d(T^{l_2}(x_i), T^{l_2}(z_2)), d(T^{l_3}(z_1), T^{l_3}(z_2))<\delta.$$
Thus, from the above claim, we get $0\leq t_1, t_2, t_3<\abs{v_n}+B$ such that $T^{l_1}(x),T^{l_1}(z_1)\in E_{n,t_1}$, $T^{l_2}(x_i), T^{l_2}(z_2)\in E_{n, t_2}$, and $T^{l_3}(z_1), T^{l_3}(z_2)\in E_{n, t_3}$. Hence, $x, z_1\in E_{n, t_1+l_1}$, $x_i, z_2\in E_{n, t_2+l_2}$, and $z_1, z_2\in E_{n, t_3+l_3}$. Letting $k_1=t_1+l_1$, $k_2=t_2+l_2$ and $k_3=t_3+l_3$, we obtain $x\in E_{n,k_1}, x_i\in E_{n, k_2}$, and both $z_1\in E_{n,k_1}\cap E_{n, k_3}\neq\emptyset$ and $z_2\in E_{n, k_2}\cap E_{n, k_3}\neq\emptyset$.

Cnversely, let $x\in X$ and fix $1\leq i\leq r$ satisfying the assumption. Let $\delta>0$ and $D$ an arbitrary open ball of radius $\frac{\delta}{3}$. By Proposition~\ref{Evnkfacts}(\ref{Evnkcontain}), there is a $E_{m,k}\subseteq D$ for some $m\in\Nn$ and $k\in\Nn$. Fix such $m$ and $k$. Note that $\mbox{diam}(E_{n,k})<\delta$ for any $n\geq m$. Let $n\geq m$ be large enough as required by the assumption.

By the assumption, there are $k_1, k_2,k_3\in\Zz$ such that $x\in E_{n,k_1}$, $x_i\in E_{n,k_2}$, and both $E_{n,k_1}\cap E_{n,k_3}\neq \emptyset$ and $E_{n,k_2}\cap E_{n,k_3}\neq \emptyset$. Let $z_1\in E_{n,k_1}\cap E_{n,k_3}$ and $z_2\in E_{n,k_2}\cap E_{n,k_3}$. Since $x,z_1\in E_{n,k_1}$, we have $T^{k_1-k}(x),T^{k_1-k}(z_1)\in E_{v_n,k}$, so letting $l_1=k_1-k$, we have $d(T^{l_1}(x),T^{l_1}(z_1))<\delta$. Similarly, letting $l_2=k_2-k$ and $l=k_3-k$, we get $d(T^{l_2}(x_i),T^{l_2}(z_2))<\delta$ and $d(T^{l}(z_1),T^{l}(z_2))<\delta$. Therefore, $(X,T)$ has partition proximality with reference points $x_1,...,x_p$.
\end{proof}

Next we tie the partition proximality property for rank-one subshifts with our analysis of different lengths of $n$-blocks.

\begin{proposition}\label{mcombinatorialproperty}
Let $p\geq 1$ be an integer. Let $(X,T)$ be a rank-one subshift with bouncede spacer parameter. Suppose for sufficiently large $n\in\Nn$ and every $h\in \Nn$, there are $n$-blocks $\alpha$, $\beta$ with $\abs{\alpha}-\abs{\beta}=ph$. Then $(X,T)$ has partition proximality with $p$ reference points.  
\end{proposition}
\begin{proof}
By Lemma \ref{boundedVpoints} there is some $a\in\Nn$ so that $X$ contains an element of the form $V^*1^aV$. Fix a such $a$ and let $x_1,...,x_p$ be points of the form $V^*1^aV$, so that for each $1\leq i\leq p$, the demonstrated occurrence of $V$ starts at position $i$ in $x_i$. In particular, $x_i[i,+\infty)=V$ and $x_i\in E_{n,i}$.

Let $n$ be sufficiently large. Let $x\in X$. Since $(X,T)$ has bounded spacer parameter, $x\in E_{n,k_1}$ for some $k_1$. Let $1\leq i\leq p$ with $i\equiv k_1\!\mod\!p$.  So $\abs{k_1-i}=hp$ for some $h$. Without loss of generality, assume $k_i-i=hp$. Our assumption gives $n$-blocks $\alpha$ and $\beta$ with $\abs{\alpha}-\abs{\beta}=hp$. Let $k_2=i$ and $k_3=k_1+\abs{\beta}$. 

By Lemma \ref{vnblockswitness}, $\beta$ witnesses that $E_{n,k_1}\cap E_{n,k_3}\neq \emptyset$. Also, $k_2+\abs{\alpha}=k_2+\abs{\beta}+hp=i+hp+\abs{\beta}=k_1+\abs{\beta}=k_3$, so $\alpha$ witnesses that $E_{n,k_2}\cap E_{n,k_3}\neq \emptyset$ by Lemma~\ref{vnblockswitness}.
Therefore, by Lemma \ref{Evnkproperty}, $(X,T)$ has partition proximality with $p$ reference points.
\end{proof}

\subsection{Maximal equicontinuous factors}

We have developed all the ingredients for the determination of the maximal equicontinuous factors for rank-one subshifts.

\begin{theorem}\label{boundedmef}
Let $(X,T)$ be a rank-one subshfit with bounded spacer parameter. Then the maximal equicontinuous factor of $(X,T)$ is $\Zz/p_{\max}\Zz$, where $p_{\max}$ is the largest $p$ with the property that  there is $n\in \Nn$ such that for all $m\geq n$ and all $1\leq i< q_m$, $p|(\abs{v_n}+a_{m,i})$.
\end{theorem}

\begin{proof} By Proposition~\ref{maxfinite}, $\Zz/p_{\max}\Zz$ is an equicontinuous factor of $(X, T)$. To see that it is maximal, we show that for sufficiently large $n\in\Nn$ and every $h\in\Nn$, there are $n$-blocks $\alpha$, $\beta$ such that $\abs{\alpha}-\abs{\beta}=hp_{\max}$. Then we apply Propositions~\ref{mcombinatorialproperty} and \ref{notafactorfinite}. 

For notational simplicity, we let $p=p_{\max}$. Since $(X,T)$ has bounded spacer parameter, there is an $n_0$ such that for all $m\geq n_0$ and $1\leq i<q_m$, $a_{m,i}$ occurs infinitely often in the spacer parameter. Fix $n>n_0$ such that for all $m\geq n$ and $1\leq i<q_m$, $p|(\abs{v_n}+a_{m,i})$. Enumerate all spacers beyond the $n$-th level by $a_1,...,a_k$. Let $d$ be the up-down gcd of $\{a_1,...,a_k\}$. By Lemma~\ref{updowngcd}, we may rewrite $\{a_1, \dots, a_k\}$ as $\{a+l_1d,...,a+l_kd\}$ for some $a,l_1,...,l_k\in \Nn$. Since $p|(\abs{v_n}+a_i)$ for each $i\leq k$, we also have $p|d$ and $p|(\abs{v_n}+a)$. Let $l\in\Nn$ be such that $\abs{v_n}+a=lp$.

Note that $\OP{gcd}(l,\frac{d}{p})=1$. Otherwise, there is some $j>1$ with $j|l$ and $jp|d$, and we would get that $jp|(lp+l_id)=\abs{v_n}+a_i$ for any $i\leq k$, which contradicts the maximality of $p=p_{\max}$.

By the Euclidean algorithm, for any $h\in\Nn$ we can find integers $s, t\in\Nn$ so that $s\frac{d}{p}=tl+h$. We will fix $t$ to be the smallest natural number so that this equation holds. Note that $t<\frac{d}{p}$. Otherwise, $t'=t-\frac{d}{p}$ and $s'=s-l$ would be smaller natural numbers satisfying the equation, contradictory to the choice of $t$.

Let $\gamma$ be any $n$-block with $t$ many expected occurrences of $v_n$. Note that $|\gamma|=tlp+\sum_{e=1}^t l_{j_e}d$ for some $1\leq j_1,\dots, j_t\leq k$. Therefore $\abs{\gamma}-tlp$ is a multiple of $d$. By Proposition \ref{updownachievable}, there are $n$-blocks $\alpha'$ and $\beta$ such that
$\gamma$ is an initial segment of $\alpha'$ and $\abs{\alpha'}-\abs{\beta}=sd+|\gamma|-tlp$. Write $\alpha'=\gamma\alpha$. Then $\alpha$ is an $n$-block, and
$$\abs{\alpha}-\abs{\beta}=\abs{\alpha'}-|\gamma|-\abs{\beta}=sd-tlp=hp. $$ 
This completes the proof of the theorem.
\end{proof}

Thus we have completely characterized the maximal equicontinuous factors for rank-one subshifts with bounded spacer parameter.
For rank-one subshifts with unbounded spacer parameter, it is easy to see that they have partion proximality with $1$ reference point, namely $1^\Zz$. Thus it follows that
they have trivial maximal equicontinuous factors.

\section{Weakly Mixing Rank-One Subshifts\label{weakmixing}}

Recall that a topological dynamical system $(X,T)$ is {\em weakly mixing} if for any non-empty open sets $U,V,W,Z\subseteq X$, there is some $l\in\Zz$ so that $T^l(U)\cap V\neq \emptyset$ and $T^l(W)\cap Z\neq \emptyset$. 

\begin{proposition}\label{weakmixingproperty}
Let $(X,T)$ be a rank-one subshift. Then the following are equivalent:
\begin{enumerate}
\item[(i)] $(X,T)$ is weakly mixing;
\item[(ii)] for any $n\in\Nn$ and $k_1,k_2\in\Zz$, there is $l\in \Nn$ such that $E_{n,-l}\cap E_{n,k_1}\neq \emptyset$ and $E_{n,-l}\cap E_{n,k_2}\neq \emptyset$;
\item[(iii)] for any $n\in\Nn$ and $h\in\Nn$, there are $n$-blocks $\alpha$ and $\beta$ such that $\abs{\alpha}-\abs{\beta}=h$.
\end{enumerate}
\end{proposition}
\begin{proof} We first show (i)$\Leftrightarrow$(ii). To see (i)$\Rightarrow$(ii) by contrapositive, assume (ii) fails. Fix $n\in\Nn$ and $k_1, k_2\in\Zz$ witnessing this failure. Then $U=W=E_{n,0}$, $V=E_{n,k_1}$, and $Z=E_{n,k_2}$ witness the failure of weak mixing.

Next, to show (ii)$\Rightarrow$(i), assume (ii) holds. Let $U,V,W,Z$ be nonempty open. By Proposition~\ref{Evnkfacts}(\ref{Evnkcontain}), we can find $n\in\Nn$ and $k_U, k_V, k_W, k_Z\in\Zz$ such that $E_{n,k_U}\subseteq U$, $E_{n,k_V}\subseteq V$, $E_{n,k_W}\subseteq W$, and $E_{n,k_Z}\subseteq Z$. Let $k_1=k_V-k_U$ and $k_2=k_Z-k_W$. Then there is some $l\in \Nn$ so that $E_{n,-l}\cap E_{n,k_1}\neq \emptyset$ and $E_{n,-l}\cap E_{n,k_2}\neq\emptyset$. By Proposition~\ref{Evnkfacts}(\ref{Evnkshift}), we have that $E_{n,k_U-l}\cap E_{n,k_V}\neq \emptyset$ and $E_{n,k_W-l}\cap E_{n,k_Z}\neq \emptyset$. But $E_{n,k_U-l}=T^{l}(E_{n,k_U})\subseteq T^{l}(U)$ and similarly $E_{n,k_W-l}\subseteq T^{l}(E_{n,k_W})$, so $l$ witnesses the weak mixing property.

Next we show (ii)$\Rightarrow$(iii). First assume (ii). Let $n, h\in\Nn$. Let $l\in\Nn$ be such that $E_{n, -l}\cap E_{n,0}\neq\emptyset$ and $E_{n, -l}\cap E_{n, h}\neq\emptyset$. By Lemma~\ref{vnblockswitness}, there are $n$-blocks $\alpha$ and $\beta$ such that $\abs{\alpha}=h+l$ and $\abs{\beta}=l$. Thus $\abs{\alpha}-\abs{\beta}=h$. Conversely, assume (iii). Note that (iii) implies that for any $n\in\Nn$ and $h\in\Nn$ there are arbitrarily long $\alpha$ and $\beta$ such that $\abs{\alpha}-\abs{\beta}=h$. To see this, consider $m\geq n$ sufficiently large. If $\alpha$ and $\beta$ are $m$-blocks then they are also $n$-blocks. So fix $n\in\Nn$ and $k_1, k_2\in \Zz$. Without loss of generality assume $k_1<k_2$. Let $\alpha$ and $\beta$ be $n$-blocks such that $\abs{\alpha}>k_2$ and $\abs{\alpha}-\abs{\beta}=k_2-k_1$. Let $l=\abs{\alpha}-k_2=\abs{\beta}-k_1>0$. Let $x\in X$ be such that $\alpha$ occurs at position $0$ in $x$ and let $y\in X$ be such that $\beta$ occurs at position $0$ in $y$. Then $x\in E_{n,0}\cap E_{n, k_2+l}$ and $y\in E_{n,0}\cap E_{n, k_1+l}$. Thus witness that $E_{n,-l}\cap E_{n,k_1}\neq \emptyset$ and $E_{n,-l}\cap E_{n,k_2}\neq\emptyset$.
\end{proof}

The following result characterizes weak mixing completely for rank-one subshifts with bounded spacer parameter.

\begin{theorem}\label{weaklymixingtheorem}
Let $(X,T)$ be a rank-one subshift with bounded spacer parameter. Then the following are equivalent:
\begin{enumerate}
\item[(i)] $(X,T)$ is weakly mixing;
\item[(ii)] the maximal equicontinuous factor of $(X, T)$ is trivial;
\item[(iii)] any finite factor of $(X, T)$ is trivial;
\item[(iv)] for any integer $p>1$ and $n\in \Nn$, there are $m\geq n$ and $1\leq i< q_m$ such that $p\!\!\not|(\abs{v_n}+a_{m,i})$.
\end{enumerate}
\end{theorem}
\begin{proof} By Theorem~\ref{boundedmef} (ii) and (iv) are equivalent. By Lemma~\ref{maxfinite} (iii) and (iv) are equivalent. By Propositions~\ref{weakmixingproperty} and \ref{mcombinatorialproperty}, (i) implies (iv). By the proof of Theorem~\ref{boundedmef} (iv) implies Proposition~\ref{weakmixingproperty}(iii).
\end{proof}

The equivalence between (i) and (ii) in the above theorem also follows from a general theorem regarding minimal topological dynamical systems admitting a invariant probability measure (see \cite{Auslander}). Here a rank-one subshift with bounded spacer parameter is minimal and uniquely ergodic, and therefore the general theorem applies. 

In the rest of this section we study the weak mixing property for rank-one subshifts with unbounded spacer parameter. The next result gives a sufficient condition.

\begin{proposition}\label{weaklymixingunbounded}
Let $(X,T)$ be a rank-one subshift. Suppose for arbitrarily large $n,n'$, there are $i,i'$ with $1\leq i< q_n$ and $1\leq i< q_{n'}$ such that $a_{n,i}-a_{n',i'}=1$. Then $(X,T)$ is weakly mixing. 
\end{proposition}

\begin{proof}
By Proposition~\ref{onegapdifferentlevels} the assumption implies Proposition~\ref{weakmixingproperty}(iii).
\end{proof}

The following gives another sufficient condition in terms of the density of the set of spacer parameters.

\begin{theorem}\label{measureoneweakly}
Let $(X,T)$ be a rank-one subshift with unbounded spacer parameter. If the set of all spacer parameters $\{a_{m, i}:m\in\Nn, 1\leq i<q_m\}$ is a subset of $\Nn$ with density greater than $\frac{1}{2}$, then $(X,T)$ is weakly mixing. 
\end{theorem}
\begin{proof}
Under the assumption there would be infinitely many pairs of successors $k$ and $k+1$ which appear in the spacer parameter. Since for any $m$, there are only finitely many spacers of the form $a_{n,i}$ where $n\leq m$, the hypothesis of Proposition~\ref{weaklymixingunbounded} holds.
\end{proof}

The next example shows that the density assumption of $\frac{1}{2}$ in Theorem~\ref{measureoneweakly} is sharp, namely, we give an example in which the spacer parameter gives a set of density exactly $\frac{1}{2}$ and the subshift fails to be weakly mixing.

\begin{example}\label{notweaklymixingexample}
Let $v_0=00$. For each $n\geq 0$, let 
$$v_{n+1}=v_nv_n11v_nv_n1111v_nv_n111111v_nv_n...v_nv_n1^{2\abs{v_n}}v_n.$$
So the spacers in $v_{n+1}$ have length $2i$ for each $0\leq i\leq \abs{v_n}$. It is easy to see that all spacers and all $v_n$ are of even length, and therefore all $n$-blocks are of even length. Thus by Proposition~\ref{weakmixingproperty} the subshift is not weakly mixing. However, the spacer set is exactly the set of all even integers, and hence has density $\frac{1}{2}$.
\end{example}

It will be useful later to note that in this construction each $v_{n+1}$ contains an odd number of expected occurrences of $v_n$.

Note that this example violates condition (iv) of Theorem~\ref{weaklymixingtheorem}, which is a non-divisibility condition for the weak mixing property for rank-one subshifts with bounded spacer parameter. We can also ask whether it is possible for a rank-one subshift with unbounded spacer parameter to satisfy this non-divisibility condition  and yet still fail to be weakly mixing. In the following we give such an example.

\begin{example}\label{mainexample}
Let $\langle\cdot,\cdot\rangle:\Nn\times\Nn\to \Nn$ be a bijection satisfying that $\langle n,l\rangle=m$ implies $n\leq m$. Let $(\cdot)_0, (\cdot)_1:\Nn\to\Nn$ be such that for all $m\in\Nn$, $\langle (m)_0,(m)_1\rangle=m$. Let $(p_n)$ enumerate all the primes. Define $v_0=0$ and 
$$ v_{n+1}=v_nv_n1^{a_n}v_n $$
where $a_n$ is the least $a>3 \abs{v_n}$ such that $a\equiv 1\ \mbox{mod}\ p_{(n)_1}$.
\end{example}

In the rest of this section we prove the claimed properties of this subshift. For clarity we will denote this subshift as $Z$, but will use the standard notation for cutting and spacer parameters. Note that for all $n$, $q_n=3$, $a_{n,1}=0$ and $a_{n,2}=a_n$.

\begin{lemma}\label{divisibilitycondition}
The subshift $Z$ satisfies that for any integer $p>1$ and $n\in\Nn$, there are $m\geq n$ and $1\leq i<q_m$ such that $p\!\!\not|(\abs{v_n}+a_{m,i})$.
\end{lemma}

\begin{proof} Otherwise there is a prime $p>1$ and $n\in\Nn$ such that for all $m\geq n$, $p|\abs{v_n}$ and $p|(\abs{v_n}+a_{m})$. Then $p|a_{m}$ for all $m\geq n$. Let $m=\langle n,p\rangle$. Then $m\geq n$ and $a_{m}\equiv 1\ \mbox{mod}\ p$, contradicting $p|a_{m}$.
\end{proof}

\begin{lemma}\label{vndifference}
Let $n\in\Nn$ and let $\alpha,\beta$ be $n$-blocks for the rank-one subshift $Z$. Suppose $\abs{\alpha}>\abs{\beta}$. Then $\abs{\alpha}-\abs{\beta}\geq \abs{v_n}$.
\end{lemma}
\begin{proof} For each $m\geq n$, we refer to exactly $a_m$ many spacers in between occurrences of $v_n$ as an {\em $m$-gap}. Thus an $m$-gap is an occurrence of $1^{a_m}$ in between two expected occurrences of $v_n$. Any $n$-block is a concatenation of disjoint expected occurrences of $v_n$ and $m$-gaps for $m\geq n$. For any $n$-block $\alpha$ and $m\geq n$, let $N_\alpha(m)$ be the number of $m$-gaps which occur in $\alpha$. Of course, for large enough $m$, $N_\alpha(m)=0$.

To prove the lemma, let $\alpha$, $\beta$ be $n$-blocks with $\abs{\alpha}>\abs{\beta}$. If $N_\alpha(m)=N_\beta(m)$ for all $m\geq n$, $\alpha$ and $\beta$ must have different numbers of expected occurrences of $v_n$. Thus $\abs{\alpha}-\abs{\beta}\geq \abs{v_n}$. Otherwise, suppose $n'\geq n$ is the largest such that $N_\alpha(n')\neq N_\beta(n')$. Thus for $m>n'$ we still have $N_\alpha(m)=N_\beta(m)$. 

Let $m_0$ be the largest $m>n'$ such that $N_\alpha(m)=N_\beta(m)\neq 0$. Note that $N_\alpha(m_0)\leq 3$ by the maximality of $m_0$. If $N_\alpha(m_0)\geq 2$, then between any two consecutive $m_0$-gaps there must be an expected occurrence of $v_{m_0}v_{m_0}v_{m_0}$. Let $\alpha_0$ be the part of $\alpha$ that is before the first $m_0$-gap, and let $\alpha_1$ be the part of $\alpha$ after the last $m_0$-gap. Then $\alpha_0$ is an end segment of $v_{m_0}v_{m_0}$ and $\alpha_1$ is an initial segment of $v_{m_0}$. Let $\beta_0$ and $\beta_1$ be similarly defined. Now consider $\alpha'=\alpha_0\alpha_1$ and $\beta'=\beta_0\beta_1$. Then $\alpha'$ and $\beta'$ are $n$-blocks, $\abs{\alpha'}-\abs{\beta'}=\abs{\alpha}-\abs{\beta}$, and for all $m\geq n$, $N_{\alpha'}(m)-N_{\beta'}(m)=N_{\alpha}(m)-N_{\beta}(m)$. In particular, we still have $\abs{\alpha'}>\abs{\beta'}$, $N_{\alpha'}(n')\neq N_{\beta'}(n')$ and for all $m>n'$, $N_{\alpha'}(m)=N_{\beta'}(m)$. Of course, $N_{\alpha'}(m_0)=N_{\beta'}(m_0)=0$. 

By repeating the construction in the above paragraph, we may assume that $N_\alpha(m)=N_{\beta}(m)=0$ for all $m>n'$. In fact, we may even apply the construction at the $n'$-th level to remove the smaller number of $n'$-gaps in $\alpha$ and $\beta$. Thus we may assume that $N_\alpha(n')=0$ or $N_\beta(n')=0$. 

If $N_\alpha(n')=0$ and $N_\beta(n')>0$, then $\abs{\alpha}\leq 3\abs{v_{n'}}<a_{n'}<\abs{\beta}$, contrary to our assumption. Thus $N_{\beta}(n')=0$ and $N_{\alpha}(n')>0$. Again $\abs{\beta}\leq 3\abs{v_{n'}}$. Since $\alpha$ is an $n$-block, we have $\abs{\alpha}\geq \abs{v_n}+a_{n'}>\abs{v_n}+3\abs{v_{n'}}$. Thus $\abs{\alpha}-\abs{\beta}\geq \abs{v_n}$.
\end{proof}

By Proposition~\ref{weakmixingproperty}(iii), $Z$ is not weakly mixing. In summary, we have shown that $Z$ satisfies the conditions (ii), (iii), and (iv) in Theorem~\ref{weaklymixingtheorem} but it fails to be weakly mixing. Thus for rank-one subshifts with unbounded spacer parameter, these conditions do not characterize weak mixing.

\section{Mixing Rank-One Subshifts\label{mixing}}

Recall that a topological dynamical system $(X,T)$ is {\em mixing} if for any non-empty open sets $U,V\subseteq X$, there is an $L\in \Nn$ such that for any $l\geq L$, $T^l(U)\cap V\neq \emptyset$.

\begin{proposition}\label{mixingproperty}
Let $(X,T)$ be a rank-one subshift. Then the following are equivalent:
\begin{enumerate}
\item[(i)] $(X,T)$ is mixing;
\item[(ii)] for any $n\in\Nn$ and $k\in \Zz$, there is $L\in \Nn$ such that for any $l\geq L$, $E_{n,-l}\cap E_{n,k}\neq\emptyset$;
\item[(iii)] for any $n\in\Nn$, there is $H\in \Nn$ such that for any $h\geq H$, there is an $n$-block $\alpha$ with $\abs{\alpha}=h$.
\end{enumerate}
\end{proposition}

\begin{proof} (i)$\Rightarrow$(ii): Suppose $(X, T)$ is mixing. Let $n\in\Nn$ and $k\in \Zz$. Consider $U=E_{n,0}$ and $V=E_{n, k}$. Let $L\in\Nn$ be such that for any $l\geq L$, $T^l(U)\cap V\neq\emptyset$. Then $E_{n,-l}\cap E_{n,k}\neq\emptyset$.

(ii)$\Rightarrow$(iii): Assume (ii) holds and let $n\in\Nn$. Then there is $H\in\Nn$ such that for any $h\geq H$, $E_{n,-h}\cap E_{n,0}\neq\emptyset$. By Lemma~\ref{vnblockswitness} there is an $n$-block $\alpha$ of length $\abs{\alpha}=h$ for all $h\geq H$.

(iii)$\Rightarrow$(i): Assume (iii) and let $U, V\subseteq$ be non-empty open sets. There is a large enough $n\in\Nn$ and $k_1, k_2\in\Nn$ such that $E_{n,k_1}\subseteq U$ and $E_{n, k_2}\subseteq V$. Let $H\in\Nn$ be such that for all $h\geq H$, there is an $n$-block $\alpha$ with $\abs{\alpha}=h$. Let $L\in\Nn$ be such that $L\geq H-k_2+k_1$. Then for all $l\geq L$, $l+k_2-k_1\geq H$, and Lemma~\ref{vnblockswitness} gives $E_{n,0}\cap E_{n, k_2-k_1+l}\neq\emptyset$. It follows that $T^l(E_{n, k_1})\cap E_{n,k_2}\neq\emptyset$, and $T^l(U)\cap V\neq\emptyset$.
\end{proof}

The following result shows that rank-one subshifts with bounded spacer parameter are never mixing. 

\begin{theorem}\label{mixingneverholds}
Let $(X,T)$ be a rank-one subshift with bounded spacer parameter. Then $(X,T)$ is not mixing.
\end{theorem}
\begin{proof}
Let $B$ be a bound on the spacer parameter and let $n\in\Nn$ be so that $\abs{v_n}>B+1$. Toward a contradiction, assume that $(X,T)$ is mixing. Then By Proposition~\ref{mixingproperty}(ii), there is some $L\in \Nn$ so that for all $l\geq L$, $E_{n,-l}\cap E_{n,0}\neq \emptyset$, or equivalently $E_{n,0}\cap E_{n,l}\neq \emptyset$.

Let $m\geq n$ be sufficiently large so that $\abs{v_m}>L$. Then by Lemma \ref{vkmisses}, $E_{n,0}\cap E_{n,\abs{v_m}+B+1}=\emptyset$. But $\abs{v_m}+B+1>\abs{v_m}>L$, which contradicts that $E_{n,0}\cap E_{n,l}\neq \emptyset$ for all $l\geq L$.
\end{proof}

In the rest of this section we consider rank-one subshifts with unbounded spacer parameter. We first give an easy example of a mixing subshift, which is an topological analog of the staircase transformation defined by Ornstein \cite{Ornstein}.

\begin{example} Define $v_0=0$ and
$$v_{n+1}=v_nv_n1v_n11v_n111v_n...v_n1^{\abs{v_n}}v_n.$$ 
Thus $q_n=\abs{v_n}+2$ and $a_{n,i}=i-1$ for all $1\leq i\leq \abs{v_n}+1$. We claim that this subshift is mixing by Proposition~\ref{mixingproperty}(iii). In fact, for any $n\in\Nn$, let $H=\abs{v_n}$. Then for all $h\geq H$ there is an $n$-block $\alpha$ with $\abs{\alpha}=h$. To see this, let $m\geq n$ be sufficiently large such that $\abs{v_m}>h$. Then there is $1\leq i<q_m$ with $a_{m,i}=h-\abs{v_n}$. Consider the $m$-block $v_m1^{a_{m,i}}$. This is also an $n$-block. Let $\alpha$ be the end segment of $v_m1^{a_{m,i}}$ starting with the last expected occurrence of $v_n$. Then $\alpha$ is an $n$-block of the form $v_n1^{h-\abs{v_n}}$. It is obvious that $\abs{\alpha}=h$. 
\end{example}

In general, let $A=\{a_{n,i}:n\in\Nn, 1\leq i<q_n\}$ be the set of all entires of the spacer parameter sequence. Then mixing can be guaranteed by appropriate largeness conditions on $A$ as in the following two results.

\begin{proposition}\label{tailimpliesmixing}
Let $(X,T)$ be a rank-one subshift. Suppose the set $A=\{a_{n,i}:n\in\Nn, 1\leq i<q_n\}$ contains a tail of $\Nn$, i.e. there is $M\in\Nn$ such that for all $a>M$, $a\in A$. Then $(X,T)$ is mixing.
\end{proposition}
\begin{proof}
Fix $n\in\Nn$. Since there are only finitely many spacer parameters of the form $a_{n',i}$ where $n'<n$, there is $M$ such that for all $a>M$, there is $m\geq n$ and $1\leq i<q_m$ with $a_{m,i}=a$. Let $H$ be such that $H> \abs{v_n}+M$. Then for any $h\geq H$, let $m\geq n$ and $1\leq i<q_m$ be such that $a_{m,i}=h-\abs{v_n}>M$. There is an $m$-block of the form $v_m1^{a_{m,i}}$. This $m$-block is also an $n$-block. Let $\alpha$ be the end segment of this $m$-block starting with the last expected occurrence of $v_n$. Then $\alpha$ is of the form $v_n1^{h-\abs{v_n}}$, and hence $\abs{\alpha}=h$.
\end{proof}

\begin{proposition}\label{increasinggapsmixing}
Let $(X,T)$ be a rank-one subshift. Let $(b_k: k\in\Nn)$ enumerate the elements of the set $\Nn\setminus A=\Nn\setminus \{a_{n,i}:n\in\Nn, 1\leq i<q_n\}$ in the increasing order. If $\liminf_k (b_{k+1}-b_k)\to\infty$, then $(X,T)$ is mixing.
\end{proposition}
\begin{proof}
Fix $n\in\Nn$. Let $A_n=\{a_{m,i}:m> n, 1\leq i<q_m\}$. Since $A\backslash A_n$ is finite, we assume without loss of generality that $A=A_n$ by only considering large enough elements. This guarantees that for any $a\in A$, $v_n1^a$ is an $n$-block.

Since $\liminf_k (b_{k+1}-b_k)\to \infty$, there is $K$ such that for all $k>K$, $b_{k+1}-b_k>\abs{v_n}+a_{n,q_{n}-1}$. Note that elements of $\Nn\backslash A$ which are no smaller than $b_K$ cannot be within $\abs{v_n}+a_{n,q_{n}-1}$ of each other.

Let $H=2\abs{v_n}+a_{n,q_n-1}+b_K$. We show that for any $h\geq H$, there is an $n$-block $\alpha$ with $\abs{\alpha}=h$. For this let $h\geq H$. If $h-\abs{v_n}\in A$, then $v_n1^{h-\abs{v_n}}$ is an $n$-block with length $h$. If $h-\abs{v_n}\not\in A$, then $h-2\abs{v_n}-a_{n,q_n-1}\in A$ since $h-2\abs{v_n}-a_{n,q_n-1}\geq b_K$. But then the string $v_n1^{a_{n,q_n-1}}v_n1^{h-2\abs{v_n}-a_{n,q_n-1}}$ is an $n$-block and has length $h$.
\end{proof}

In the next two results the largeness of $A$ is relaxed but still enough to guarantee mixing.

\begin{proposition}\label{arithmeticsequencemixing}
Let $(X,T)$ be a rank-one subshift. Suppose that 
\begin{enumerate}
\item $A=\{a_{n,i}: n\in\Nn, 1\leq i<q_n\}$ contains an arithmetic sequence $(ks+t: k\in\Nn)$ for some $s>0$ and $t\geq 0$; and 
\item for every $n$, there are $m_1,\dots, m_s\geq n$ such that $\abs{v_{m_j}}\equiv j\ \mbox{mod}\ s$ for all $1\leq j\leq s$.
\end{enumerate}
Then $(X,T)$ is mixing.
\end{proposition}
\begin{proof}
Fix $n\in\Nn$. Let $m_1, \dots, m_s$ be as in assumption (2). Let $M=\max\{ m_j:1\leq j\leq s\}$ and $L=\max\{\abs{v_{m_j}}:1\leq j\leq s\}$. Let $H$ be large enough so that all terms $ks+t>H-L$ occur as $a_{m,i}$ for some $m\geq M$ and $1\leq i<q_m$. Let $h\geq H$. Then $h-t\equiv \abs{v_{m_j}}\ \mbox{mod}\ s$ for some $1\leq j\leq s$. So $h-t=ks+\abs{v_{m_j}}$ for some $k$, so $h=ks+t+\abs{v_{m_j}}$. Since $h\geq H$, $ks+t\in A$ and $ks+t> H-L$, and thus it occurs as $a_{m,i}$ for some $m\geq m_j$ and $1\leq i<q_m$.
Let $\alpha=v_{m_j}1^{ks+t}$. Then $\alpha$ is an $n$-block and $\abs{\alpha}=h$.
\end{proof}

Note that although the conditions in Proposition~\ref{arithmeticsequencemixing} are technical, it is not hard to construct such rank-one generating sequences by diagonalization. The following result is a generalization of Proposition~\ref{arithmeticsequencemixing} with a similar proof. We state it without proof.

\begin{proposition}\label{arithmeticsequencegeneral}
Let $(X,T)$ be a rank-one subshift. Suppose that
\begin{enumerate}
\item $A=\{a_{n,i}: n\in\Nn, 1\leq i<q_n\}$ contains an arithmetic sequence $(ks+t: k\in\Nn)$ for some $s>0$ and $t\geq 0$; and 
\item for every $n$, there are $n$-blocks $w_1,\dots, w_s$ each of which is an end segment of some $v_m$ where $m\geq n$, and for each $1\leq j\leq s$, $\abs{w_j}\equiv j\ \mbox{mod}\ s$. 
\end{enumerate}
Then $(X,T)$ is mixing.
\end{proposition}

Again, it is easy to construct rank-one subshifts satisfying these very flexible conditions. Note that in all four propositions above the set $A=\{a_{n,i}: n\in\Nn, 1\leq i<q_n\}$ has positive density. We do not have an example of a mixing rank-one subshift where the set $A$ has density $0$. Also, we do not know if the set $A$ having density $1$ implies mixing for rank-one subshifts.

In the rest of this section we construct two examples of rank-one subshifts which have the same set of spacer parameters but exhibit different mixing properties. In fact, one of them is mixing, while the other one is not even weakly mixing. These examples show that the mixing properties cannot be determined by the set of spacer parameters alone.

\begin{example} Fix an integer $p\geq 2$. Define $v_0=0$ and
$$v_{n+1}=v_nv_n1^pv_n1^{2p}v_n1^{3p}v_n\cdots v_n1^{(\abs{v_n}-1)p}v_n. $$
Denote this rank-one subshift by $X_p$. We will show that $X_p$ is mixing. 

Alternatively, define $u_0=0^p$ and 
$$u_{n+1}=u_nu_n1^pu_n1^{2p}u_n1^{3p}u_n\cdots u_n1^{(\abs{u_n}-1)p}u_n. $$
Denote this rank-one subshift by $Y_p$. Then $Y_p$ is not weakly mixing. In fact, since every term of the spacer parameter sequence is a multiple of $p$, it is easy to show by induction that $\abs{u_n}$ is a multiple of $p$ for all $n\in\Nn$. It follows that the length of any $n$-block is a multiple of $p$, and in particular cannot be arbitrary when sufficiently large.
\end{example}

\begin{lemma} $X_p$ is mixing. \end{lemma}
\begin{proof}
A straightforward induction gives that for every $n$, $\abs{v_n}\equiv 1\ \mbox{mod}\ p$. Fix $n\in\Nn$. Let $\abs{v_n}=Np+1$. Let $H=Np^2+p^3$. We show that for all $h\geq H$, there is an $n$-block $\alpha$ with length $h$. For this let $h\geq H$ and $j\equiv h\ \mbox{mod}\ p$ where $1\leq j\leq p$. Write $h=j+kp$ and let 
$$ a=k-jN-\displaystyle\frac{(j-2)(j-1)}{2}. $$
Since $h\geq H$, we have $a\geq 0$. Consider
$$ \alpha=v_n1^{ap}v_nv_n1^pv_n1^{2p}\cdots v_n1^{(j-2)p}. $$
The first expected occurrence of $v_n$ is an end segment of some $v_m$ for appropriate $m\geq n$. Starting from the second expected occurrence of $v_n$ is an initial segment of $v_{n+1}$ that includes $j-1$ many expected occurrences of $v_n$. Then $\alpha$ is an $n$-block with length
$$ j\abs{v_n}+ap+\displaystyle\frac{(j-2)(j-1)}{2}p=jNp+j+kp-jNp=h. $$
\end{proof}

Note that $X_p$ and $Y_p$ have essentially the same cutting and spacer parameters. These examples show that a complete classification of weak mixing or mixing for rank-one subshifts with unbounded spacer parameters would need to be sensitive to changes even at the lowest level.

\end{document}